\documentclass[a4paper,reqno,10pt]{amsart}
\usepackage {amsmath,amssymb}
\usepackage{amsfonts}
\usepackage{amsthm}
\usepackage{tikz}
\usepackage[encapsulated]{CJK}
\usepackage{graphicx,color}
\usepackage [latin1]{inputenc}
\usepackage {times}
\usepackage {float}
\usepackage{hyperref}
\usepackage{amscd}
\usepackage[shortalphabetic,nobysame]{amsrefs}

\usetikzlibrary{matrix,arrows}

\numberwithin{equation}{section}
\numberwithin{figure}{section}

\newcommand{\ldz}{L_{2}^{3}}
\newcommand{\R}{\mathbb R}

\newcommand{\Z}{\mathbb Z}

\newcommand{\N}{\mathbb N}

\DeclareMathOperator{\Sym}{Sym}
\newcommand{\betrag}[1]{\left|#1\right|}
\DeclareMathOperator{\PPP}{PP}
\DeclareMathOperator{\POO}{P}
\DeclareMathOperator{\MM}{L}
\DeclareMathOperator{\LL}{L}
\DeclareMathOperator{\Star}{Star}
\newcommand{\PPk}[1]{\PPP^{k}(#1)}
\newcommand{\PPaffk}[1]{\PPP^{k}(#1)}
\newcommand{\PPO}[2]{\PPP^{#1}(#2)}
\newcommand{\PPOaff}[2]{\PPP^{#1}(#2)}
\newcommand{\MPPOaff}[2]{\MM\PPP^{#1 -1}(#2)}
\newcommand{\PO}[2]{\POO^{#1}(#2)}
\newcommand{\MPP}[1]{\MM\PPP^{k-1}(#1)}
\newcommand{\MPPaff}[1]{\MM\PPP^{k-1}(#1)}
\newcommand{\PPa}[1]{\PPP^{\ast}(#1)}

\newcommand{\Zaffd}{Z^{\text{\rm{fan}}}_{d}}
\newcommand{\Zaff}[1]{Z^{\text{\rm{fan}}}_{#1}}

\newcommand{\Zy}{Z}
\newcommand{\AAA}{C}
\newcommand{\Aa}[1]{\AAA^{\ast}(#1)}
\newcommand{\Ak}[2]{\AAA^{#1}(#2)}
\newcommand{\Zk}[2]{\Zy_{#1}(#2)}

\DeclareMathOperator{\Div}{Div}
\newcommand{\curlyx}{\mathcal{X}}
\newcommand{\curlyy}{\mathcal{Y}}

\newcommand{\curlyc}{\mathcal{C}}
\DeclareMathOperator{\trop}{B}
\DeclareMathOperator{\cl}{cl}
\DeclareMathOperator{\rank}{r}
\DeclareMathOperator{\Att}{A}
\newcommand{\AT}{\Att_{\text{T}}}

\newcommand{\arxiv}[1]{%
  \href{http://arxiv.org/abs/#1}{arxiv:#1}%
}

\newtheorem{theorem}{Theorem}[section]
\newtheorem{lemma}[theorem]{Lemma}
\newtheorem{defthm}[theorem]{Definition and Theorem}
\newtheorem{proposition}[theorem]{Proposition}
\newtheorem{corollary}[theorem]{Corollary}

\theoremstyle{definition}
\newtheorem{definition}[theorem]{Definition}
\newtheorem{example}[theorem]{Example}

\newtheorem{notation}[theorem]{Notation}
\newtheorem{defconst}[theorem]{Definition and Construction}
\theoremstyle{remark}
\newtheorem{remark}[theorem]{Remark}
\newtheorem*{acknowledgement}{Acknowledgement}

\numberwithin{equation}{section}

\begin{document}
\title{Cocycles on tropical varieties via piecewise polynomials}
\author {Georges Francois}
\address {Georges Francois, Fachbereich Mathematik, Technische Universit\"{a}t
  Kaiserslautern, Postfach 3049, 67653 Kaiserslautern, Germany}
\email {gfrancois@email.lu}
\begin{abstract}
We use piecewise polynomials to define tropical cocycles generalising the well-known notion of tropical Cartier divisors to higher codimensions. Groups of  cocycles are tropical analogues of Chow cohomology groups. We also introduce
an intersection product of cocycles with tropical cycles - the counterpart of the classical cap product - and prove that this gives rise to a Poincar\'e duality in some cases.
\end{abstract}
\maketitle

\section{Introduction}

In \cite{katzpayne} the authors 
describe a method to assign a Minkowski weight in a complete fan $\Delta$ to a piecewise polynomial on $\Delta$ and therefore suggest to use piecewise polynomials in tropical geometry. If $\Delta$ is unimodular (i.e.\ corresponds to a smooth toric variety), their assignment is even an isomorphism. 

We show in the third section that the assignment of \cite{katzpayne} - which describes the canonical map from equivariant to ordinary Chow cohomology rings of the corresponding toric variety - agrees with the (inductive) intersection product of rational functions introduced in \cite{AR}. This motivates us to use piecewise polynomials as local ingredients for tropical cocycles. It turns out that each piecewise polynomial on an arbitrary tropical fan is a sum of products of rational functions; this can be used to intersect cocycles with tropical cycles. One should note that, in contrast to the classical cap product, our intersection product is well-defined on the level of cycles - not only on classes modulo rational equivalence. We finish the section by deducing a Poincar\'e duality on the cycle $\R^n$ from the isomorphism between the groups of piecewise polynomials and Minkowski weights on complete unimodular fans.

In the fourth section we focus on matroid varieties and smooth tropical varieties (i.e.\ cycles which locally look like matroid varieties). Thereby we prove that each subcycle of a matroid variety (modulo lineality space) can be cut out by a cocycle (theorem \ref{surjective}) and show a Poincar\'e duality in codimension $1$ and dimension $0$ for smooth varieties (corollary \ref{poincaresmooth}). 

A similar construction to piecewise polynomials on fans has recently been introduced independently in \cite{esterov}: Esterov defines tropical varieties with (degree $k$) polynomial weights  and their (codimension $1$) corner loci which are tropical varieties with (degree $k-1$) polynomial weights.

I would like to thank Andreas Gathmann and Johannes Rau for many helpful discussions.
\begin{acknowledgement}
The author is supported by the Fonds national de la Recherche (FNR), Luxembourg.
\end{acknowledgement}

\section{Tropical cycles and rational functions}
In this section we briefly recall the notions of tropical cycles and rational functions as well as the intersection product of rational functions with tropical cycles introduced in \cite{AR}.

A pure-dimensional rational polyhedral complex in a real vector space $V=\Lambda\otimes_{\Z}\R$ associated to a lattice $\Lambda$ is a rational polyhedral complex all of whose maximal polyhedra have the same dimension. The set of $k$-dimensional polyhedra in the polyhedral complex $\curlyx$ is denoted by $\curlyx^{(k)}$. We say that the pure-dimensional rational polyhedral complex $\curlyx$ in $V$ is weighted if it is equipped with a weight function $\omega_{\curlyx}:\curlyx^{(\dim\curlyx)}\rightarrow\Z$ on its maximal polyhedra. Each polyhedron $\sigma\in\curlyx$ induces a linear subspace $V_{\sigma}$ of $V$ generated by differences of vectors in $\sigma$ as well as a sublattice $\Lambda_{\sigma}:=\Lambda\cap V_{\sigma}$. Recall that the notation $\sigma>\tau$ means that the polyhedron $\tau$ is a proper face of the polyhedron $\sigma$; if $\tau$ is a codimension $1$ face of $\sigma$, then we denote the primitive normal vector of $\sigma$ relative to $\tau$ by $u_{\sigma/\tau}$. A tropical polyhedral complex $\curlyx$ of dimension $d$ in $V$ is a weighted rational polyhedral complex of pure dimension $d$ in $V$ which satisfies the following balancing condition for each $\tau\in\curlyx^{(d-1)}$:
\[\sum_{\sigma\in\curlyx:\sigma>\tau} \omega_{\curlyx}(\sigma)\cdot u_{\sigma/\tau}=0\in V/V_{\tau}.\]  
A (tropical) cycle in $V$ is the equivalence class modulo refinement of a tropical polyhedral complex. If $\curlyx$ is a representative of the tropical cycle $X$, then we call $\curlyx$ a polyhedral structure of $X$; in this situation we also call $X$ the cycle associated to $\curlyx$. The support $|X|$ of a tropical cycle $X$ is the union of polyhedra of non-zero weight in a polyhedral structure of $X$. A cycle $Y$ is a subcycle of $X$ if $|Y|\subseteq |X|$; the (additive) group of $k$-dimensional subcycles of $X$ is denoted by $\Zy_k(X)$. A fan cycle $X$ is a cycle associated to a tropical fan (which is then called fan structure of $X$). The group of $k$-dimensional fan subcycles of a fan cycle $X$ is denoted $\Zaff{k}(X)$. The cycle $X$ is irreducible if $\Zy_{\dim X}(X)=\Z\cdot X$; it is locally irreducible if multiples of the balancing condition are the only linear relations of the primitive normal vectors $u_{\sigma/\tau}$. A morphism $f:X\rightarrow Y$ of tropical cycles is a locally integer affine linear map $|X|\rightarrow |Y|$ on the supports. If the morphism $f$ has an inverse which is a morphism, and $\omega_{\curlyx}(\sigma)=\omega_{\curlyy}(f(\sigma))$ for all $\sigma\in\curlyx$ and suitable polyhedral structures $\curlyx,\curlyy$ of $X,Y$, then $f$ is an isomorphism.  We refer to \cite[section 2, section 5]{AR} for more details about tropical cycles.

A rational function on a cycle $X$ is a piecewise affine linear function $\varphi:|X|\rightarrow\R$; that means there is a polyhedral structure $\curlyx$ of $X$ such that the restriction of $\varphi$ to each polyhedron $\sigma\in\curlyx$ is a sum $\varphi_{|\sigma}=\varphi_{\sigma}+a_{\sigma}$ of an integer linear form $\varphi_{\sigma}\in \Lambda_{\sigma}^{\vee}$ and a real constant $a_{\sigma}$. A rational fan function on a fan cycle $X$ is a rational function which is linear on the cones of some fan structure of $X$. 

Let $\varphi$ be a rational function on the cycle $X$ and let $\curlyx$ be a polyhedral structure of $X$ such that $\varphi$ is affine linear on every polyhedron of $\curlyx$. 
The weighted polyhedral complex $\varphi\cdot \curlyx$ is defined as the polyhedral complex $\curlyx\setminus\curlyx^{(\dim X)}$ together with the weight function (on the codimension $1$ polyhedra in $\curlyx$)
\[\omega_{\varphi\cdot\curlyx}(\tau):= \sum_{\sigma\in\curlyx:\sigma>\tau} \omega_{\curlyx}(\sigma)\cdot \varphi_{\sigma}(v_{\sigma/\tau})-\varphi_{\tau}\left(\sum_{\sigma\in\curlyx:\sigma>\tau}\omega_{\curlyx}(\sigma)\cdot v_{\sigma/\tau}\right),\]
where the $v_{\sigma/\tau}\in V$ are representatives of the primitive normal vectors $u_{\sigma/\tau}\in V/V_{\tau}$.
The associated cycle is denoted by $\varphi\cdot X\in \Zy_{\dim X-1}(X)$. If $\varphi$ is an affine linear function, then $\varphi\cdot X=0$. As restrictions of rational functions are again rational functions one can define the intersection product of $\varphi$ with a subcycle $Y$ of $X$ as $\varphi_{\mid Y}\cdot Y$. Note that intersections of rational fan functions with fan cycles are again fan cycles. More details about the intersection product with rational functions can be found in \cite[section 3]{AR}.

\section{Piecewise polynomials and tropical cocycles}
In order to fix our notations we start this section by recalling the definition of a piecewise polynomial on a (not necessarily tropical) fan $F$.

\begin{definition}
Let $\sigma$ be a (rational) cone in the vector space $V=\Lambda\otimes_{\Z}\R$ corresponding to a lattice $\Lambda$. We define $\PO{k}{\sigma}$ to be the set of functions $g:\sigma\rightarrow\R$ that extend to a homogeneous polynomial of degree $k$ on the subspace $V_{\sigma}$ having integer coefficients (that means $g\in\Sym^k (\Lambda_{\sigma}^{\vee})$ - the degree $k$ part of the symmetric algebra $\Sym (\Lambda_{\sigma}^{\vee})$). 
A piecewise polynomial of degree $k$ on a (rational) fan $F$ in $V$ is a continuous function $f:\betrag{F}\rightarrow \R$ on the support of $F$ such that the restriction $f_{\mid \sigma}\in\PO{k}{\sigma}$ for each cone $\sigma\in F$. The sum of two degree $k$ piecewise polynomials $f,g$ on $F$ is defined pointwise, i.e.\ $(f+g)(x):=f(x)+g(x)$. As the $\PO{k}{\sigma}$ are additive groups, the sum $f+g$ is again a piecewise polynomial on $F$. The (additive) group of piecewise polynomials of degree $k$ on the fan $F$ is denoted by $\PPk{F}$. Since products of homogeneous integer polynomials are again homogeneous integer polynomials, the pointwise multiplication of two piecewise polynomials $f\in\PPk{F},g\in\PPO{l}{F}$ is in $\PPO{k+l}{F}$. We call $\PPa{F}:=\oplus_{k\in\N}\PPk{F}$ the graded ring of piecewise polynomials on $F$. Finally, we define $\MPP{F}:=\langle\{l\cdot f: l \text{ linear}, f\in\PPO{k-1}{F}\}\rangle$ to be the subgroup of $\PPk{F}$ generated by linear functions.
\end{definition}

\begin{remark}
Note that piecewise polynomials of degree $1$ are the same as rational fan functions and that restrictions of piecewise polynomials to subfans are again piecewise polynomials (of the same degree).
\end{remark}

\begin{notation}
We denote by $\Zy_{k}(\curlyx)$ the group of $k$-dimensional Minkowski weights in a tropical fan $\curlyx$ (i.e.\ its elements are $k$-dimensional tropical subfans of $\curlyx$). 
\end{notation}

We are ready to state the above-mentioned result of Katz and Payne which was proved in chapter 1, proposition 1.2, theorem 1.4 of \cite{katzpayne}.
\begin{defthm}
\label{katzpayne}
Let $\Delta$ be a complete unimodular (i.e.\ every cone is generated by a part of a lattice basis) fan in $\R^n$. If $\{v_1,\ldots,v_n\}$ is a basis of $\Z^n$, we can regard the elements $v_1^*,\ldots,v_n^*$ of its dual basis as integer linear functions on $\R^n$, that means $v_i^*\in \PO{1}{\R^n}$.  For two cones $\tau<\sigma\in\Delta^{(n)}$, with $\sigma$ generated by $v_1,\ldots,v_n$, one can thus define $e_{\sigma,\tau}:=\prod_{i: v_i\not\in\tau}\frac{1}{v_i^{\ast}}$ to be the inverse of the product $\prod_{v_i\notin\tau}v_i^*\in\PO{n-\dim(\tau)}{\R^n}$. Let $f\in\PPk{\Delta}$. For a maximal cone $\sigma\in\Delta$, $f_{\sigma}$ denotes the polynomial on $\R^n$ which agrees with $f$ on $\sigma$. Then for any $\tau\in\Delta^{(n-k)}$
\[c_{f\cdot\Delta}(\tau):=\sum_{\sigma\in\Delta^{(n)}: \sigma>\tau}e_{\sigma,\tau}f_{\sigma}\]
 is an integer. Furthermore,  \[f\cdot\Delta:=\left(\bigcup_{i\leq n-k}\Delta^{(i)},c_{f\cdot\Delta} \right)\] is a tropical fan, and the assignment  \[\PPk{\Delta}/\MPP{\Delta}\rightarrow\Zy_{n-k}(\Delta), \ \ f\mapsto f\cdot\Delta\] is an isomorphism of groups. The fan cycle associated to $f\cdot\Delta$ is independent under refinement of $\Delta$ and is denoted by $f\cdot\R^n$.
\end{defthm}

\begin{remark}
\label{toric}
All previous notions have counterparts in toric intersection theory: Groups of piecewise polynomials $\PPk{F}$ on any fan $F$ are canonically isomorphic to equivariant Chow cohomology groups $\AT^k(X(F))$ of the associated toric variety $X(F)$ \cite[theorem 1]{payne}. For any complete $n$-dimensional fan $\Delta$, the groups $\Zy_{n-k}(\Delta)$ of Minkowski weights are canonically isomorphic to (ordinary) Chow cohomology groups $\Att^k(X(\Delta))$ of the complete toric variety $X(\Delta)$ \cite[theorem 2.1]{FS}. Katz and Payne showed in \cite[theorem 1.4]{katzpayne} that, under these identifications, the canonical map $\AT^k(X(\Delta))\rightarrow\Att^k(X(\Delta))$ is given by intersections with piecewise polynomials $f\mapsto f\cdot \Delta$.
\end{remark}

\begin{example}
\label{f*R2}
Let $\{e_1,e_2\}$ be the standard basis of $\R^2$ and let $\Delta$ be the complete fan with maximal cones $\langle -e_1,e_1+e_2\rangle$, $\langle -e_2,e_1+e_2\rangle$ and $\langle -e_1,-e_2\rangle$. We write $x:=e_1^*$, $y:=e_2^*$ and see that the dual bases are given by $\{y,y-x\}$, $\{x,x-y\}$ and $\{-x,-y\}$ respectively. Let $f\in\PPP^{2}(\Delta)$ be the piecewise polynomial shown in the picture.
\begin{center}
\includegraphics[scale=0.16]{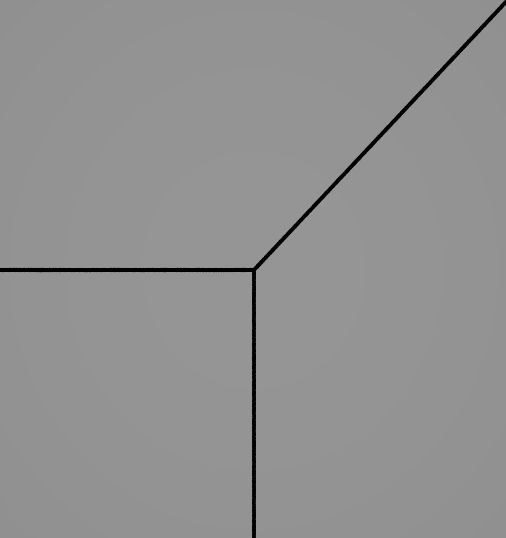}
\put(-30,25){$x^2$}
\put(-60,60){$y^2$}
\put(-60,25){$0$}
\end{center}
Then $f\cdot\Delta$ is the origin with weight
\[c_{f\cdot\Delta}(\{0\})=\frac{y^2}{y(y-x)}+\frac{x^2}{x(x-y)}+\frac{0}{(-x)(-y)}=\frac{-xy^2+yx^2}{xy(x-y)}=1.\]
Note that $f=(\max\{x,y,0\})^2=\max\{x,y\}\cdot\max\{x,y,0\}$ is a product of rational functions and that $\max\{x,y,0\}\cdot\max\{x,y,0\}\cdot\Delta$ and $\max\{x,y\}\cdot\max\{x,y,0\}\cdot\Delta$ (as product of rational functions with cycles defined in the previous section) give the origin with weight $1$ too.
\end{example}

If a piecewise polynomial on a complete fan $\Delta$ is a product of rational (i.e.\ piecewise linear) functions $\varphi_i$, then there are two ways of defining its intersection product with $\Delta$: We can either intersect inductively with the rational functions $\varphi_i$ (cf.\ section 2) or use the formula of theorem \ref{katzpayne}. In the previous example both ways led to the same result. 
We show in the following proposition that this is true in general:

\begin{proposition}
\label{equality}
Let $\Delta$ be a complete unimodular fan in $\R^n$, and let $\varphi_1,\ldots,\varphi_k$ be rational functions on $\R^n$ which are linear on every cone of $\Delta$. Let $f=\varphi_1\cdots \varphi_k\in\PPk{\Delta}$. Then $f\cdot\R^n=\varphi_1\cdots \varphi_k \cdot\R^n$, where the products on the right hand side are products of rational functions with cycles (cf.\ section 2).
\end{proposition}

In order to prove this proposition we need the following lemma.
\begin{lemma}
\label{phi}
Let $\varphi$ be a rational fan function on a fan cycle $X$ which is linear on the cones of the unimodular fan structure $\curlyx$ of $X$. For a ray $r$ of $\curlyx$ with primitive integral vector $v_r$, let $\Psi_r$ be the function which is linear on the cones of $\curlyx$, satisfies $\Psi_r(v_r)=1$, and maps the primitive integral vectors of all other rays of $\curlyx$ to $0$. Then $\varphi=\sum_{r\in\curlyx^{(1)}}\varphi(v_r)\cdot\Psi_r$.
\end{lemma}
\begin{proof}
As $\varphi$ and $\sum_{r\in\curlyx^{(1)}}\varphi(v_r)\cdot\Psi_r$ are both linear on the cones of $\curlyx$ it suffices to compare their values on the primitive integral vectors of the rays of $\curlyx$ where they agree by construction. 
\end{proof}

\begin{proof}[Proof of proposition \ref{equality}]
Let $\tau\in\Delta^{(n-k)}$ be an arbitrary codimension $k$ cone in $\Delta$. By adding an appropriate linear function $l$, we can assume that the restriction ${\varphi_1}_{\mid\tau}$ is identically zero. This does not change $f\cdot\R^n$ since $l\cdot \varphi_2\cdots \varphi_k$ is in $\MPP{\Delta}$. In the following $r,r_i$ denote rays of $\Delta$ with respective primitive integral vector $v,v_i$. If $\alpha<\sigma$ are cones in the unimodular fan $\Delta$, then $\sigma$ is the Minkowski sum of $\alpha$ and the $\dim(\sigma)-\dim(\alpha)$ rays of $\sigma$ which are not in $\alpha$. We first assume that $k=1$. The fact that ${\varphi_1}_{\mid\tau}=0$ implies that 
\[\omega_{\varphi_1\cdot\Delta}(\tau)=\sum_{r:\tau+r\in\Delta^{(n)}}\varphi_1(v)=\sum_{r:\tau+r\in\Delta^{(n)}}\frac{1}{v^*}(\varphi_1)_{\tau+r}=c_{\varphi_1\cdot\Delta}(\tau),\]
where the sums run over Minkowski sums $\tau+r$. Here the middle equality follows from lemma \ref{phi} and the facts that ${\varphi_1}_{\mid \tau}=0$ and $(\Psi_r)_{\tau+r}=v^*$.\\
Now we assume that $k>1$ and set $g:=\varphi_2\cdots \varphi_k$. As ${\varphi_1}_{\mid\tau}=0$ the definition of intersecting with a rational function implies that $\tau$ has weight   
\[\omega_{\varphi_1\cdots \varphi_k\cdot\Delta}(\tau)=\sum_{r: \tau+r\in\Delta^{(n-k+1)}}\omega_{\varphi_2\cdots \varphi_k\cdot\Delta}(\tau+r)\cdot \varphi_1(v)\]
in $\varphi_1\cdots \varphi_k\cdot\Delta$.
By induction on the degree of $f$ this is equal to
\begin{eqnarray*}
&&\sum_{r: \tau+r\in\Delta^{(n-k+1)}} c_{(\varphi_2\cdots \varphi_k)\cdot\Delta}(\tau+r)\cdot \varphi_1(v) \\
&=&  \sum_{r: \tau+r\in\Delta^{(n-k+1)}}\sum_{\sigma>\tau+r} e_{\sigma,\tau+r}\cdot g_{\sigma}\cdot \varphi_1(v)\\
&=& \sum_{\stackrel{\sigma>\tau \text{ in } \Delta^{(n)}}{\sigma=\tau+r_1+\ldots+r_{k}}}\sum_{i=1}^{k} \varphi_1(v_i)\cdot e_{\sigma,\tau+r_i}\cdot g_{\sigma}\\
&=& \sum_{\stackrel{\sigma>\tau \text{ in } \Delta^{(n)}}{\sigma=\tau+r_1+\ldots+r_{k}}}\sum_{i=1}^{k} \varphi_1(v_i)\cdot v_i^{\ast}\cdot e_{\sigma,\tau}\cdot g_{\sigma}\\
&=& \sum_{\stackrel{\sigma>\tau \text{ in } \Delta^{(n)}}{\sigma=\tau+r_1+\ldots+r_{k}}} e_{\sigma,\tau} \cdot {\left(g\cdot\sum_{i=1}^{k}\varphi_1(v_i)v_i^{\ast}\right)}_{\sigma}.
\end{eqnarray*}
As in the induction start we use ${\varphi_1}_{\mid\tau}=0$ to conclude that the above agrees with
 \[\sum_{\stackrel{\sigma>\tau \text{ in } \Delta^{(n)}}{\sigma=\tau+r_1+\ldots+r_{k}}} e_{\sigma,\tau} \cdot {\left(g\cdot \varphi_1\right)}_{\sigma}=c_{f\cdot\Delta}(\tau).\]
\end{proof}

\begin{remark}
There is an alternative way of deducing the $k>1$ case of proposition \ref{equality} from the $k=1$ case: 
As the canonical map $\AT^*(X(\Delta))\rightarrow\Att^*(X(\Delta))$ is a ring homomorphism it follows from remark \ref{toric} that \[(\varphi_1\cdots\varphi_k)\cdot\Delta=(\varphi_1\cdot\Delta)\cup\ldots\cup(\varphi_k\cdot\Delta),\]
where the cup products on the right hand side are computed using the fan displacement rule of \cite[proposition 3.1, theorem 3.2]{FS}. By \cite[theorem 4.4]{katz} or \cite[theorem 1.9]{hannes} the cup product of Minkowski weights agrees with the intersection product of tropical cycles (cf.\ \cite[definition 9.3]{AR}) which implies that the above is equal to $\varphi_1\cdots\varphi_k\cdot\Delta$ (interpreted as inductive intersection product with rational functions). 
 
\end{remark}

So far $\R^n$ is the only fan cycle admitting an intersection product with piecewise polynomials (cf.\ theorem \ref{katzpayne}). Therefore, our next aim is to define an intersection product for arbitrary fan cycles. The idea is to write piecewise polynomials as sums of products of rational fan functions and use these representations to define an intersection product.
We introduce some more notation:

\begin{notation}
The group of piecewise polynomials of degree $k$ on a fan cycle $X$ is defined to be $\PPaffk{X}:=\{f: f\in\PPk{\curlyx} \text{ for some fan structure $\curlyx$ of $X$}\}$. We set $\MPPaff{X}:=\langle\{l\cdot f: l\text{ linear }, f\in\PPOaff{k-1}{X}\}\rangle$.
\end{notation}

\begin{notation}
Let $F$ be a unimodular fan and let $v_1,\ldots,v_m$ be the primitive integral vectors of the rays $r_1,\ldots,r_m$ of $F$. Then $\Psi_{r_i}:=\Psi_i\in\PPO{1}{F}$ is the unique function which is linear on the cones of $F$ and satisfies $\Psi_i(v_j)=\delta_{ij}$, where $\delta_{ij}$ is the Kronecker delta function. For a cone $\tau\in F$ we have a piecewise polynomial $\Psi_{\tau}:=\prod_{i:v_i\in\tau}\Psi_i\in\PPO{\dim\tau}{F}$. Note that $\Psi_{\tau}$ vanishes away from $\bigcup_{\sigma>\tau}\sigma$.
\end{notation}

As mentioned in \cite{brion}, one can show that the functions $\Psi_{\tau}$ generate the ring of piecewise polynomials:

\begin{proposition}
\label{zerlegung}
Let $f\in\PPk{F}$ be a piecewise polynomial of degree $k$ on a unimodular fan $F$. Then there exists a representation $f=\sum_{\sigma\in F^{(\leq k)}}a_{\sigma}\Psi_{\sigma}$, where the $a_{\sigma}$ are homogeneous integer polynomials of degree $k-\dim(\sigma)$ and the sum runs over all cones of $F$ of dimension at most $k$. In particular, piecewise polynomials on tropical  fan cycles are sums of products of rational functions.
\end{proposition}

\begin{proof}
We use induction on the dimension of $F$, the case $\dim F=0$ being obvious. We know by induction hypothesis that there are homogeneous integer polynomials $a_{\sigma}$ such that $f_{\mid\betrag{F_1}}=\sum_{\sigma\in F_1^{(\leq k)}}a_{\sigma}\Psi_{\sigma}$, where $ F_1:=\{\sigma:\sigma\in F^{(p)} \text{ with } p<\dim F\}$. Thus, it suffices to show the claim for $g:=f-\sum_{\sigma\in F_1^{(\leq k)}}a_{\sigma}\Psi_{\sigma}\in\PPk{F}$. Now we use induction on the number $r$ of maximal cones in $F$. Let $r=1$ and $\sigma$ be the unique maximal cone in $F$. By \cite[section 1.2]{brion}, we know that the following sequence is exact: \[0\rightarrow \Psi_{\sigma}\PO{k-\dim F}{F}\hookrightarrow\PPk{F}\stackrel{\text{rest.}}{\rightarrow}\PPk{F\setminus\{\sigma\}}\rightarrow 0.\]
Here $\PO{k-\dim F}{F}$ denotes the group of homogeneous integer polynomials of degree $k-\dim F$ on $F$. Since $g_{\mid\betrag{ F\setminus\{\sigma\}}}=g_{\mid\betrag{ F_1}}=0$, it follows that there is a polynomial $a_{\sigma}$ such that $g=a_{\sigma}\Psi_{\sigma}$. Now let $r>1$ and $\sigma\in F$ a maximal cone. By the induction hypothesis, there are polynomials $b_{\tau}$ such that $g_{\mid \betrag{F\setminus\{\sigma\}}}=\sum_{\tau\in F\setminus\{\sigma\}^{(\leq k)}}b_{\tau}\Psi_\tau$. Since the restriction of $g-\sum_{\tau\in F\setminus\{\sigma\}^{(\leq k)}}b_{\tau}\Psi_\tau$ to $F\setminus\{\sigma\}$ is $0$, the claim follows from the exactness of the above sequence.
It remains to prove the ``in particular'' statement. Let $f'\in\PPk{X}$ be a piecewise polynomial on a fan cycle $X$. We choose a a fan structure $\curlyx$ of $X$ such that $f'\in\PPk{\curlyx}$ and refine it to a unimodular fan structure $\curlyx'$ (cf.\ \cite[section 2.6]{fulton}). Now we apply the first part of the proposition to $f'\in\PPk{\curlyx'}$; as the $\Psi_{\sigma}$ are products of rational functions and the homogeneous integer polynomials $a_{\sigma}$ are sums of products of linear functions, it follows that $f'$ is a sum of products of rational functions.  
\end{proof}

The previous proposition together with the well-known intersections with rational functions enable us to define an intersection product of piecewise polynomials with tropical fan cycles. Later we will use this to construct an intersection product of cocycles with arbitrary cycles.

\begin{definition}
\label{localint}
Let $f\in\PPk{X}$ be a piecewise polynomial on a fan cycle $X\in\Zaffd(V)$. By proposition \ref{zerlegung} we can choose rational functions $\varphi_j^{i}$ such that $f=\sum_{i=1}^{s}\varphi_1^{i}\cdots\varphi_k^{i}\in\PPk{X}$. This allows us to define the intersection of $f$ with the cycle $X$ to be
\[f\cdot X:=\sum_{i=1}^{s}(\varphi_1^{i}\cdots\varphi_k^{i}\cdot X)\in\Zaff{d-k}(X).\]
In fact, we can define the intersection of $f$ with any fan subcycle of $X$ in this way. 
\end{definition}

We have seen in example \ref{f*R2} that representations of piecewise polynomials as sums of products of rational functions are not unique. Therefore, we need to ensure that the intersection product does not depend on the chosen representation: 

\begin{proposition}
\label{prod}
Let $\varphi_1^{i},\ldots,\varphi_k^{i}$, $\gamma_1^{j},\ldots,\gamma_k^{j}$, with $k\leq d$ be rational fan functions on a fan cycle $X\in\Zaffd(V)$ such that $f:=\sum_{i\in I} \varphi_1^{i}\cdots\varphi_k^{i}=\sum_{j\in J} \gamma_1^{j}\cdots\gamma_k^{j}\in\PPk{X}$. Then we have the following equation of intersection products of rational functions with cycles: \[\sum_{i\in I} \varphi_1^{i}\cdots\varphi_k^{i}\cdot X=\sum_{j\in J} \gamma_1^{j}\cdots\gamma_k^{j}\cdot X.\]
\end{proposition}

The proof of the proposition makes use of the following technical lemma:

\begin{lemma}
\label{vandermonde}
Let $c_{b_1\ldots b_s}$ be real numbers such that $\sum_{b_1+\ldots + b_s=k}a_1^{b_1}\cdots a_s^{b_s}\cdot c_{b_1\ldots b_s}=0$ for all $a_i> 0$. Then all $c_{b_1\ldots b_s}$ are $0$.
\end{lemma}
\begin{proof}
For $a_1\in\{1,\ldots,k+1\}$ and any $a_2,\ldots,a_s>0$ we have
\[0=\sum_{b_1+\ldots + b_s=k}a_1^{b_1}\cdots a_s^{b_s}\cdot c_{b_1\ldots b_s}= \sum_{b_1=0}^{k}a_1^{b_1}\sum_{b_2+\ldots+b_s=k-b_1} a_2^{b_2}\cdots a_s^{b_s}c_{b_1\ldots b_s}.\]
Since the Vandermonde matrix ${(i^j)}_{i=1,\ldots,k+1,j=0,\ldots,k}$ is regular, it follows that \[\sum_{b_2+\ldots+b_s=k-b_1} a_2^{b_2}\cdots a_s^{b_s}c_{b_1\ldots b_s}=0\] for all $a_2,\ldots,a_s> 0$ and all $b_1\in\{0,\ldots,k\}$. Hence the claim follows by induction.
\end{proof}

\begin{proof}[Proof of proposition \ref{prod}]
We choose a unimodular fan structure $\curlyx$ of $X$ such that all $\varphi_p^{i},\gamma_p^{j}$ are linear on every cone of $\curlyx$. Let $v_1,\ldots,v_m$ be the primitive integral vectors of the rays $r_1,\ldots,r_m$ of $\curlyx$. By lemma \ref{phi} $\varphi_p^{i}=\sum_{s=1}^{m}\varphi_p^{i}(v_s)\cdot\Psi_s$ and we have 
\begin{eqnarray*}
f=\sum_{i\in I} \varphi_1^{i}\cdots\varphi_k^{i}&=&\sum_{i\in I}\left(\sum_{s=1}^{m}\varphi_1^{i}(v_s)\cdot\Psi_s\right)\cdots\left(\sum_{s=1}^{m}\varphi_k^{i}(v_s)\cdot\Psi_s\right)\\ &=& \sum_{i\in I}\sum_{\ 1\leq s_1\leq\ldots\leq s_k\leq m \ }\sum_{\sigma\in S_k}\varphi_1^{i}(v_{s_{\sigma(1)}})\cdots\varphi_k^{i}(v_{s_{\sigma(k)}})\cdot  \Psi_{s_1}\cdots\Psi_{s_k}\\ &=& \sum_{1\leq s_1\leq\ldots\leq s_k\leq m \ }\underbrace{\sum_{\sigma\in S_k}\sum_{i\in I}\varphi_1^{i}(v_{s_{\sigma(1)}})\cdots\varphi_k^{i}(v_{s_{\sigma(k)}})}_{=:\lambda_{s_1\ldots s_k}\in\Z}\cdot  \Psi_{s_1}\cdots\Psi_{s_k},
\end{eqnarray*}
where the middle sums run over all permutations $\sigma\in S_k$.
The linearity and commutativity of intersecting with rational functions \cite[proposition 3.7]{AR} allow us to perform the same computation for the intersection product with $X$; thus we have  \[\sum_{i\in I} \varphi_1^{i}\cdots\varphi_k^{i}\cdot X=\sum_{1\leq s_1\leq\ldots\leq s_k\leq m}\lambda_{s_1\ldots s_k}\cdot\Psi_{s_1}\cdots\Psi_{s_k}\cdot X.\] 

Analogously we set $\mu_{s_1\ldots s_k}:=\sum_{\sigma\in S_k}\sum_{j\in J}\gamma_1^{j}(v_{s_{\sigma(1)}})\cdots\gamma_k^{j}(v_{s_{\sigma(k)}})$
and use the same argument for the $\gamma_p^j$ to conclude that
\[
\sum_{i\in I} \varphi_1^{i}\cdots\varphi_k^{i}\cdot X-\sum_{j\in J} \gamma_1^{j}\cdots\gamma_k^{j}\cdot X = \sum_{1\leq s_1\leq\ldots\leq s_k\leq m}\underbrace{(\lambda_{s_1\ldots s_k}-\mu_{s_1\ldots s_k})}_{=:c_{s_1\ldots s_k}}\cdot\Psi_{s_1}\cdots\Psi_{s_k}\cdot X.
\]
As $\Psi_{w_1}\cdots\Psi_{w_k}\cdot X=0$ if the cone $\langle w_1,\ldots,w_k\rangle\notin\curlyx$ (this can be showed in the same way as \cite[lemma 1.4]{lars}) the above is equal to
\[\sum_{\langle v_{s_1},\ldots ,v_{s_k}\rangle\in\curlyx}c_{s_1\ldots s_k}\cdot\Psi_{s_1}\cdots\Psi_{s_k}\cdot X.\]
Note that the $s_i$ are not necessarily pairwise disjoint; that means the sum runs over all cones in $\curlyx$ of dimension at most $k$.
It suffices to prove that all $c_{s_1\ldots s_k}$ occurring in the above sum are equal to $0$: Therefore, we fix integers $1\leq t_1<\ldots < t_k\leq m$ such that $\langle v_{t_1},\ldots, v_{t_k}\rangle\in\curlyx^{(k)}$ and claim that $c_{s_1\ldots s_k}=0$ for all $1\leq s_1\leq\ldots\leq s_k\leq m$ with $\{s_1,\ldots,s_k\}\subseteq\{t_1,\ldots,t_k\}$: For all $a_1,\ldots,a_k > 0$ we have 
\begin{eqnarray*}
0&=&(f-f)(a_1 v_{t_1}+\ldots + a_k v_{t_k})\\&=&\left (\sum_{1\leq s_1\leq\ldots\leq s_k\leq m}c_{s_1\ldots s_k}\cdot\Psi_{s_1}\cdots\Psi_{s_k}\right )(a_1 v_{t_1}+\ldots + a_k v_{t_k})\\ &=& \sum_{\stackrel{1\leq s_1\leq\ldots\leq s_k\leq m}{\{s_1,\ldots,s_k\}\subseteq\{t_1,\ldots,t_k\}}}c_{s_1\ldots s_k}\prod_{i=1}^{k}a_i^{\betrag{\{j:s_j=t_i\}}}\\ &=&\sum_{b_1+\ldots+b_k=k}c_{\underbrace{t_1\ldots t_1}_{b_1 \text{ times}}\ldots\underbrace{t_k\ldots t_k}_{b_k \text{ times}}}a_1^{b_1}\cdots a_k^{b_k},
\end{eqnarray*}
where the last sum runs over non-negative integers $b_i$ that sum up to $k$. Now the claim follows from lemma \ref{vandermonde}. 
\end{proof}

\begin{remark}
\label{easy}
It is clear that the intersection product of definition \ref{localint} is linear and that $f\cdot (g\cdot X)=(f\cdot
g)\cdot X$ for two piecewise polynomials $f,g$ on a fan cycle $X$. Furthermore, it
follows straight from definition that $f\cdot X=0$ if $f\in\MPP{X}$.
\end{remark}

\begin{example}
Let $f\in\PPP^{2}(\ldz)$ be the piecewise polynomial on the tropical fan cycle $\ldz:=\max\{x,y,z,0\}\cdot\R^3$ shown in the following picture.  
Let $\curlyx$ be the corresponding fan structure of $\ldz$ having rays $(-1,0,0)$, $(-1,-1,0)$, $(0,-1,0)$, $(0,0,-1)$, $(1,1,0)$, $(1,1,1)$.

\begin{tabular}{p{5.5cm}p{0.1cm}p{5.5cm}}
\parbox[c]{1em}{\includegraphics[scale=0.245]{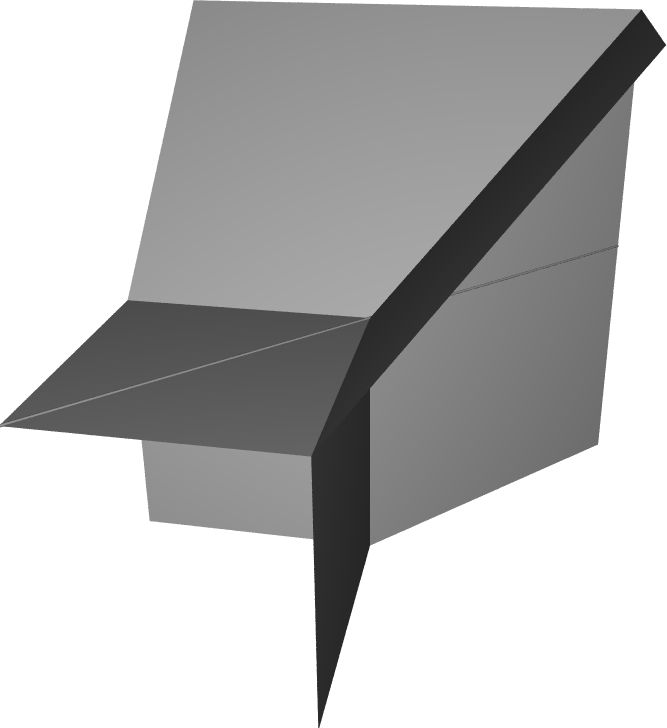}}
\put(70,50){$xy$}
\put(119,50){\textcolor{white}{$x^2$}}
\put(126,33){$xz$}
\put(104,-10){$2xz$}
\put(69,-35){\textcolor{white}{$yz$}}
\put(35,-30){$xz$}
\put(38,-10){$2x^2$}
\put(15,6){$y^2+xy$}
&&
\parbox[c]{1em}{\includegraphics[scale=0.245]{l32.jpg}}
\put(70,50){$0$}
\put(119,50){\textcolor{white}{$0$}}
\put(126,33){$0$}
\put(104,-10){$2xz$}
\put(69,-35){\textcolor{white}{$yz$}}
\put(35,-30){$xz$}
\put(38,-10){$0$}
\put(15,6){$y^2-xy$}
\put(-4,6){$\sigma_1$}
\put(15,-30){$\sigma_2$}
\put(56,-70){$\sigma_3$}
\put(104,-37){$\sigma_4$}
\\
\center{$f\in\PPP^2(\curlyx)\subsetneq\PPP^2(\ldz)$} && \center{$f-2x\cdot\Psi_a-x\cdot\Psi_b$}
\end{tabular}

We want to compute $f\cdot\ldz$. Therefore, we use the idea of the proof of proposition \ref{zerlegung} to obtain a representation of $f$ as a sum of products of rational
functions: We first make $f$ vanish on the rays of $\curlyx$ by adding appropriate (linear) multiples of the rational functions $\Psi_r$ (with $r$ ray of $\curlyx$). Doing this we obtain $f-2x\cdot\Psi_a-x\cdot\Psi_b$, where $a=(-1,-1,0)$ and $b=(1,1,1)$. Now it is easy to see that
\[f-2x\cdot\Psi_a-x\cdot\Psi_b=-\Psi_{\sigma_1}+\Psi_{\sigma_2}+\Psi_{\sigma_3}-2\cdot \Psi_{\sigma_4}.\]
As an easy calculation shows that $\Psi_{\sigma_i}\cdot\ldz=1\cdot\{0\}$ for all $i$ (alternatively use lemma \ref{Psisigma}) we obtain by definition \ref{localint} and remark \ref{easy} that $f\cdot\ldz=(-1+1+1-2)\cdot\{0\}=-1\cdot\{0\}$.
\end{example}

\begin{remark}
\label{locality}
Let $X$ be a tropical cycle in the vector space $V$. Let $p$ be a point in $X$. Recall that in \cite[section 1.6]{hannes} the star $\Star_X(p)$ is defined to be the 
tropical fan cycle in $V$ associated to $\Star_{\curlyx}(p)$, where $\curlyx$ is a polyhedral structure of $X$ containing the polyhedron $\{p\}$. That means $\Star_X(p)$ is the fan cycle whose support consists of vectors $v$ such that $p+\epsilon v\in\betrag{X}$ for small (positive) $\epsilon$ and whose weights are inherited from $X$.\\
A piecewise polynomial $f\in\PPk{X}$ on a fan cycle $X$ induces a piecewise polynomial $f^p\in\PPk{\Star_X(p)}$ obtained by restricting $f$ to a small neighbourhood of $p$ and then extending it in the obvious way to $\Star_X(p)$. As $f=\sum_{i=1}^{s}\varphi_1^{i}\cdots\varphi_k^{i}$ implies that $f^p=\sum_{i=1}^{s}(\varphi_1^{i})^p\cdots(\varphi_k^{i})^p$, it follows from \cite[proposition 1.1]{hannes} that \[f^p\cdot \Star_X(p)=\Star_{f\cdot X}(p).\]
\end{remark}

Our next aim is to use piecewise polynomials to define higher codimension cocycles on tropical cycles $X$. Prior to that we give a definition of (abstract) tropical cycles consistent with the 
definition of smooth tropical varieties in chapter 6 of \cite{francoisrau} (to which we refer for further details). Recall that a topological space is called weighted if each point from a dense open subset is equipped with a non-zero integer weight which is locally constant (in the dense open subset). A cycle $X$ in a vector space can be made weighted by assigning to each interior point of a maximal polyhedra $\sigma\in\curlyx$ the weight of $\sigma$, where $\curlyx$ is a polyhedral structure of $X$.

\begin{definition}
\label{defcycle}
An (abstract) tropical cycle is a weighted topological space $X$ together with an 
    open cover $\{U_i\}$ and homeomorphisms
  \[
    \phi_i : U_i \rightarrow W_i \subseteq |X_i|
  \]
  such that
  \begin{itemize} 
  \item
    each $W_i$ is an (euclidean) open subset of $|X_i|$ for some tropical fan cycle $X_i$ (in some vector space)
  \item
    for each pair $i,k$, the transition map
    \[
      \phi_k \circ \phi_i^{-1} :
        \phi_i(U_i \cap U_k) \rightarrow \phi_k(U_i \cap U_k)
    \]
    is the restriction of an affine $\Z$-linear map, i.e.\
    the composition of a translation by a real vector and a $\Z$-linear map
   \item the weight of a point $p\in U_i$ is equal to the weight of $\phi_i(p)$ in $X_i$ (if both are defined).

\end{itemize}
If all $X_i$ can be chosen to be matroid varieties (cf.\ section 4 or \cite[section 2]{francoisrau}) modulo lineality spaces, then we call $X$ a smooth tropical variety. 
Recall that in \cite[definition 6.2]{francoisrau} a subcycle $C$ of $X$ is defined as a weighted subset of $\betrag{X}$ such that 
for all $i$ the induced weighted set $\phi_i(C\cap U_i)$ agrees with the intersection of $W_i$ and a tropical cycle in $X_i$.
\end{definition}

\begin{definition}
\label{PPU}
Let $X$ be a fan cycle (in a vector space $V$) and $U$ an (euclidean) open subset in $\betrag{X}$. A continuous function $f:U\rightarrow \R$ is called piecewise polynomial of degree $k$ on $U$ if it is locally around each point
$p\in U$ a finite sum $\sum_j (f_p^j\circ T_p^j)$ of compositions of (restrictions of) piecewise polynomials $f_p^j\in\PPk{\Star_X(p)}$ and translations $T_p^j$. We define $f_p\in\PPk{\Star_X(p)}$ to be the (uniquely defined) sum of the $f_p^j$.  The group of piecewise polynomials of degree $k$ on $U$ is denoted $\PPk{U}$. Furthermore, $\MPP{U}$ is the group of piecewise polynomials $f$ (of degree $k$) on $U$ such that $f_p\in\MPP{\Star_X(p)}$ for all $p$.
\end{definition}

We now generalise the notion of Cartier divisors (i.e.\ codimension $1$ cocycles) introduced in \cite[definition 6.1]{AR} by using piecewise polynomials (instead of piecewise linear functions) as local descriptions:

\begin{definition}
\label{cocycle}
A representative of a codimension $k$ cocycle on the cycle $X$ is defined as a set $\{(V_1,f_1),\ldots,(V_p,f_p)\}$ satisfying
\begin{itemize}
\item $\{V_i\}$ is an open cover of $\betrag{X}$
\item $(f_j\circ \phi_i^{-1})_{\mid \phi_i(U_i\cap V_j)}\in\PPk{\phi_i(U_i\cap V_j)}$ for all $i,j$
\item $((f_j-f_k)\circ \phi_i^{-1})_{\mid \phi_i(U_i\cap V_j\cap V_k)}\in\MPP{\phi_i(U_i\cap V_j\cap V_k)}$ for all $i,j,k$.
\end{itemize}
The sum of two (representatives of) codimension $k$ cocycles $\{(V_j,f_j)\}$ and $\{(V'_k,f'_k)\}$ is defined to be $\{(V_j\cap V'_k),f_j+f'_k)\}$.
We call two representatives of codimension $k$ cocycles $\{(V_j,f_j)\}$ and $\{(V'_k,f'_k)\}$ equivalent (and identify them) if we have for all $i,s$ that \[(g_s\circ \phi_i^{-1})_{\mid \phi_i(U_i\cap K_s)}\in\MPP{\phi_i(U_i\cap K_s)},\] where $\{(K_s,g_s)\}:=\{(V_j,f_j)\}-\{(V'_k,f'_k)\}$.\\
The group of codimension $k$ cocycles on $X$ is denoted $\Ak{k}{X}$. The multiplication of two cocycles can be defined in the same way as the addition; therefore, there is a graded ring $\Aa{X}:=\oplus_{k\in\N}\Ak{k}{X}$ called ring of piecewise polynomials.
\end{definition}

\begin{example}
For any cycle $X$, $\Ak{1}{X}$ is the group of Cartier divisors $\Div(X)$ introduced in \cite[definition 6.1]{AR}.
\end{example}

\begin{example}
Vector bundles $\pi:F\rightarrow X$ of degree $r$ on tropical cycles $X$ have been introduced in \cite[definition 1.5]{vectorbundles}. A rational section $s:X\rightarrow F$ with open cover $U_1,\ldots,U_s$ induces rational functions $s_{ij}:=p_j^{(i)}\circ\Phi_i\circ s:U_i\rightarrow \R$ (cf. \cite[definition 1.18]{vectorbundles}). Here the $\Phi_i$ are homeomorphisms identifying $\pi^{-1}(U_i)$ with $U_i\times\R^r$ and the $p_j^{(i)}:U_i\times\R^r\rightarrow \R$ are projections to the $j$-th component of $\R^r$. For any $k\leq r$ one obtains the cocycle $s^{(k)}:=\{(U_i,\sum_{1\leq j_1\leq\ldots\leq j_k\leq r}s_{ij_1}\cdots s_{ij_k})\}\in\Ak{k}{X}$ (see \cite[definition 2.1]{vectorbundles}).
\end{example}

We are now ready to construct an intersection product of cocycles with tropical cycles. As cocycles are locally given by piecewise polynomials, the idea is to glue together the local intersection products of definition \ref{localint}.   

\begin{defconst}
\label{cocycle*cycle}
Let $f=\{(V_j,f_j)\}\in\Ak{k}{X}$ be a codimension $k$ cocycle on a tropical cycle $X$.
For a point $p$ in $X$ we choose $i,j$ such that $p\in U_i\cap V_j$.  
By definition $(f_j\circ\phi_i^{-1})_p\in\PPk{\Star_{X_i}(\phi_i(p)}$ is a piecewise polynomial on the star around $\phi_i(p)$. Thus we can define the local intersection $(f_j\circ\phi_i^{-1})\cdot (X_i\cap\phi_i(U_i\cap V_j))$
by \[\Star_{(f_j\circ\phi_i^{-1})\cdot (X_i\cap\phi_i(U_i\cap V_j))}(\phi_i(p)):=(f_j\circ\phi_i^{-1})_p\cdot \Star_{X_i}(\phi_i(p)).\]
As $\phi_k\circ\phi_i^{-1}$ induces an isomorphism of the stars $\Star_{X_i}(\phi_i(p))$ and $\Star_{X_k}(\phi_k(p))$, the definition does not depend on the choice of open set $U_i$. \\
We can glue together the local intersections to a subcycle $f\cdot X\in\Zk{\dim X-k}{X}$ of $X$: If $p\in U_i\cap V_j\cap V_s$, then  $((f_j-f_s)\circ\phi_i^{-1})_p\in\MPP{\Star_{X_i}(\phi_i(p)}$. Therefore, it follows by remark \ref{easy} that the local intersections agree on the overlaps.
\end{defconst}

\begin{remark}
\label{remintprod}
In the same way we can also intersect cocycles on $X$ with any subcycle of $X$. Hence, definition \ref{cocycle*cycle} gives rise to an intersection product
\[\Ak{k}{X}\times\Zk{l}{X}\rightarrow\Zk{l-k}{X}, \ \ (f,C) \mapsto f\cdot C.\]
\end{remark}

\begin{example}
The following picture shows a cocycle $f=\{(V_1,f_1),(V_2,f_2)\}\in\Ak{2}{\R^2}$,
\begin{center}
\includegraphics[scale=0.25]{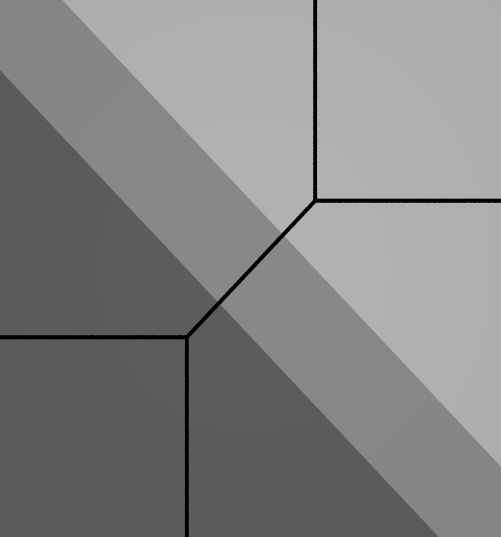}
\put(-110,5){\textcolor{white}{\small{$V_1$}}}
\put(-15,120){\small{$V_2$}}
\put(-30,14){\small{$V_1\cap V_2$}}
\put(-100,30){\textcolor{white}{\small{$0$}}}
\put(-70,30){\textcolor{white}{\small{$(x+1)^2$}}}
\put(-102,70){\textcolor{white}{\small{$(y+1)^2$}}}
\put(-30,100){\small{$0$}}
\put(-91,110){\small{$(x-2)^2$}}
\put(-45,65){\small{$(y-2)^2$}}
\put(-87,41){\small{$R$}}
\put(-56,84){\small{$Q$}}
\end{center}
with $R=(-1,-1)$, $Q=(2,2)$. Note that for $p=(t,t)$ with $ -1< t < 2$ we have \[(f_1-f_2)_p=(f_1)_p-(f_2)_p=(\max\{x,y\})^2-(\min\{x,y\})^2=(y+x)\cdot\max\{x-y,y-x\},\] 
which is in $\LL\PPP^1(\Star_p(\R^2))$ (cf.\ definition \ref{PPU}); hence $f$ is indeed a cocycle. 
As $(f_1)_R$ is the piecewise polynomial of example \ref{f*R2} we conclude that the multiplicity of $R$ in $f\cdot\R^2$ is $1$. We can deduce from an analogous argument for the point $Q$ that $f\cdot\R^2=R+Q$. 
\end{example}

As in the case of rational functions and Cartier divisors \cite[proposition 4.7 and 7.6]{AR}, we can pull back piecewise polynomials and cocycles along morphisms.

\begin{definition}
If $h:Y\rightarrow X$ is a morphism of fan cycles and $f\in\PPaffk{X}$ is a piecewise polynomial on $X$, then we define the pull-back $h^{\ast}f\in\PPaffk{Y}$ of $f$ along the morphism $h$ as $h^{\ast}f:=f\circ h$.\\
Analogously the pull-back $h^{\ast}f\in\Ak{k}{Y}$ of a codimension $k$ cocycle $f=\{(V_j,f_j)\}\in\Ak{k}{X}$ along a morphism $h:Y\rightarrow X$ of any cycles (cf.\ \cite[remark 6.5]{francoisrau}) is defined to be $\{(h^{-1}(V_j),f_j\circ h)\}$.
\end{definition}

\begin{proposition}
The following properties hold for cocycles $f\in\Ak{k}{X}$ and $g\in\Ak{l}{X}$ on a cycle $X$.
\begin{enumerate}
\item $\Ak{k}{X}\times\Zk{l}{X}\rightarrow\Zk{l-k}{X},\ \ (b,C) \mapsto b\cdot C$ is bilinear.
\item $f\cdot (g\cdot X)=(f\cdot g)\cdot X=g\cdot(f\cdot X)$.
\item $f\cdot(h_{\ast}E)=h_{\ast}(h^{\ast}f\cdot E)$ for a morphism $h:Y\rightarrow X$ and a subcycle $E$ of $Y$.
\item If $X\in\Zk{d}{V_X},Y\in\Zk{e}{V_Y}$ are cycles in vector spaces, then $(f\cdot X)\times Y=\pi^{\ast}f\cdot(X\times Y)$, where $\pi:V_X\times V_Y\rightarrow V_X$ maps $(x,y)$ to $x$.
\item If $D$ is rationally equivalent to $0$ on $X$ \cite[definition 1]{AR2}, then so is $f\cdot D$. 
\end{enumerate}
If $X$ and $Y$ are smooth and $C,D$ are subcycles of $X$, then intersection products and pull-backs (\cite[definition 6.4 and 8.1]{francoisrau}) have the following additional properties:
\begin{enumerate}
\setcounter{enumi}{5}
\item If $D=f\cdot X$, then $D\cdot_{X} C=f\cdot C$.
\item If $b$ is a cocycle on $C$, then $(b\cdot C)\cdot D=b\cdot (C\cdot D)$.
\item If $D=f\cdot X$ and $h:Y\rightarrow X$ is a morphism, then $h^{\ast}D=h^{\ast}f\cdot Y$.
\end{enumerate}
\end{proposition}

\begin{proof}
We first notice that all statements except (5) can be verified locally (that means for piecewise polynomials on fan cycles). But the local statements are either trivial or
follow directly from the respective property of rational functions (\cite[4.8, 9.6, 9.7, 9.8]{AR}, \cite[4.5, 8.2]{francoisrau}). Using (3) the proof of (5) is the same
as the proof of \cite[lemma 2(b)]{AR2}. 
\end{proof}

For the rest of the chapter we focus on cocycles on the cycle $\R^n$. We use theorem \ref{katzpayne} to 
establish a Poincar\'e duality for this case:

\begin{theorem}
\label{ausschneiden}
For any $n\geq k$, the following is a group isomorphism: \[\Ak{k}{\R^n}\rightarrow\Zk{n-k}{\R^n}, \ \ f\mapsto f\cdot\R^n.\]
\end{theorem}
\begin{proof}
We first consider the corresponding local statement: Since every fan cycle in $\R^n$ has a fan structure lying in a complete unimodular fan (\cite[lemma 5]{AR2} and \cite[section 2.6]{fulton}), we can use theorem \ref{katzpayne} to conclude that \[\PPaffk{\R^n}/\MPPaff{\R^n}\rightarrow\Zaff{n-k}(\R^n), \ \ \ g\mapsto g\cdot\R^n\] is an isomorphism.  \\
For the global case we start by proving the surjectivity. So let $C\in\Zk{n-k}{\R^n}$ be an arbitrary subcycle of $\R^n$ and let $\curlyc$ be a polyhedral structure of $C$. We choose an open cover $\{V_j\}$ of $\R^n$ and translation functions $T_j$ such that $T_j(\curlyc\cap V_j)$ is a tropical open fan (that is the set of intersections of the cones of a fan with an euclidean open set which contains the origin, cf.\ \cite[definition 5.3]{AR}) for all $j$. By the local statement we can choose for each $j$ a piecewise polynomial $f_j$  whose intersection with $\R^n$ is the tropical fan associated to $T_j(\curlyc\cap V_j)$. Then $f=\{(V_j,f_j\circ T_j)\}\in\Ak{k}{\R^n}$ is a cocycle satisfying $f\cdot\R^n=C$. By construction the difference of two of these local functions gives a zero intersection on the overlaps of two open sets, i.e.\ the local intersection $(f_j\circ T_j-f_k\circ T_k)\cdot (\R^n\cap V_1\cap V_2)=0$; therefore, the injectivity part of the local statement implies that the third condition of definition \ref{cocycle} is fulfilled and $f$ is indeed a cocycle on $\R^n$. \\
The injectivity follows immediately from the local statement.
\end{proof}

\begin{remark}
Let $Y$ be a codimension $1$ fan subcycle of the fan cycle $X$ and let $\curlyy,\curlyx$ be simplicial fan structures of $Y,X$ such that $\curlyy$ is a subfan of $\curlyx$. It is easy to see from the definition of intersecting with rational functions in section 2 that finding a rational function $\varphi$ which is linear on the cones of $\curlyx$ and satisfies $\varphi\cdot\curlyx=\curlyy$ can be reduced to solving a system of linear equations with variables $\varphi(v),v\in\curlyx^{(1)}$. Unless $X=\R^n$ these systems are in no way guaranteed to have a (rational) solution. In case that $X=\R^n$, $\varphi$ can also be found by inductively determine the linear functions $\varphi_{\sigma}$ along a path of adjacent maximal cones of $\curlyx$ \cite[proof of corollary 2.4]{FS}.\\ 
Except in a few special cases (cf.\ remark \ref{rn-rm}), we do not have an algorithm to construct a concrete piecewise polynomial on $\R^n$ which cuts out a given fan subcycle $Y$ of codimension greater than $1$. 
\end{remark}

\section{Cocycles on matroid varieties}

In this section we analyse cocycles on smooth varieties. As mentioned in definition \ref{defcycle} a tropical cycle is smooth if its local building blocks are matroid varieties modulo lineality spaces (denoted by $\trop(M)/L$). 

Let $M$ be a (loopfree) matroid with ground set $E=\{1,\ldots,n\}$ (see \cite{matroidtheory} for standard notation about matroids). Recall that we can associate a tropical fan $\mathcal{B}(M)$ to the matroid $M$ in the following way: Let $\{e_1,\ldots,e_n\}$ be the standard basis of $\R^n$. For each flat (i.e.\ closed set) $F$ of $M$ we set $V_F=-\sum_{i\in F} e_i$. The tropical fan $\mathcal{B}(M)$ consists of cones 
\[\left\{\sum_{i=1}^{p}\lambda_i\cdot V_{F_i}: \lambda_1,\ldots,\lambda_{p-1}\geq 0, \lambda_p\in\R\right\},\]
where $\emptyset\subsetneq F_1\subsetneq\ldots\subsetneq F_{p-1}\subsetneq F_p=E$ is a chain of flats in $M$ and all of the maximal cones have trivial weight $1$. The associated tropical fan cycle is denoted $\trop(M)$ and called matroid variety. Matroid varieties have a natural lineality space $\R\cdot(e_1+\ldots +e_n)$.

The following theorem states that every subcycle of a matroid variety can be cut out by a cocycle. The idea of the proof is to delete elements of the matroid in order to make use of the $\R^n$ case: If an element $i$ of $E$ is not a coloop, then the deletion of $i$ \cite[section 1.3, 3.1.5]{matroidtheory} corresponds to a projection. This means that 
the push-forward of $\trop(M)$ along the projection $\pi_i:\R^n\rightarrow \R^{n-1}$ forgetting the $i$-th coordinate is equal to the matroid variety $\trop(M\setminus\{i\})$ corresponding to the deletion matroid \cite[lemma 3.8]{francoisrau}.

\begin{theorem}
\label{surjective}
For any $k\leq d:= \dim (\trop(M)/L)$, the following morphism is surjective: \[\Ak{k}{\trop(M)/L}\rightarrow\Zk{d-k}{\trop(M)/L}, \ \ f\mapsto f\cdot\trop(M)/L.\]
\end{theorem}
\begin{proof}
We first consider the case where $L=\{0\}$ and $\{a\}$ is a flat for every $a\in E$. We use induction on the codimension of $\trop(M)$: The induction start ($\trop(M)=\R^n$) was proved in theorem \ref{ausschneiden}. Let $C$ be an arbitrary subcycle of $\trop(M)$ of codimension $k$. After renaming the elements, we can assume that $\{1,\ldots,p\}$ is the set of elements of $E$ which are not coloops. For $i\in\{1,\ldots,p\}$ we set
\[C_0:=C, \ \ \ C_i:=C_{i-1}-\pi_{i}^{\ast}{\pi_{i}}_{\ast}C_{i-1},\]
where the $\pi_i:\trop(M)\rightarrow\trop(M\setminus\{i\})$ denote the projections forgetting the $i$-th coordinate. 
The induction hypothesis allows us to choose cocycles $f_i\in\Ak{k}{\trop(M\setminus\{i\})}$ such that $f_i\cdot\trop(M\setminus\{i\})={\pi_i}_{\ast}C_{i-1}$ for $i\in\{1,\ldots,p\}$.  
\cite[Lemma 9.3]{francoisrau} implies that ${\pi_i}_{\ast}C_p=0$ for all $i$; thus $C_p=0$ by \cite[lemma 9.4]{francoisrau}. It follows that
\[C=\sum_{i=1}^p \pi_i^{\ast}{\pi_i}_{\ast}C_{i-1}=\sum_{i=1}^p \pi_i^{\ast}(f_i\cdot\trop(M\setminus\{i\}))=\sum_{i=1}^p (\pi_i^{\ast}f_i)\cdot\trop(M).\]
As $\pi_R:\trop(M)\rightarrow \trop(M\setminus R)$ is an isomorphism for $R=\cl_M(\{a\})\setminus \{a\}$, where $\cl_M(\{a\})$ denotes the smallest flat in $M$ containing the set $\{a\}$, this also implies the claim for arbitrary matroid varieties $\trop(M)$.\\
Now let $C$ be a subcycle of $\trop(M)/L$. Since $\trop(M)\cong\trop(M)/L\times L$ we can choose a cocycle $f$ with $f\cdot(\trop(M)/L\times L)=C\times L$. It follows that $f\cdot(\trop(M)/L\times\{0\})=C\times\{0\}$. Therefore, we can conclude that $s^{\ast}f\cdot\trop(M)/L=C$, where $s:\trop(M)/L\rightarrow\trop(M)/L\times L$ maps $x$ to $(x,0)$.
\end{proof}

\begin{remark}
It follows in the same way that each fan cycle $D\in\Zaff{d-k}(\trop(M)/L)$ is cut out by a piecewise polynomial $f\in\PPaffk{\trop(M)/L}$. 
\end{remark}

\begin{remark}
An alternative proof (in the case of a trivial lineality space $L=\{0\}$) has recently been found by Esterov in \cite[corollary 4.2]{esterov}.
\end{remark}

\begin{remark}
\label{rn-rm}
If $N,M$ are matroids such that $\trop(N)$ is a codimension $k$ subcycle of $\trop(M)$, then \cite[corollary 3.6, proposition 3.10]{francoisrau} gives a concrete piecewise polynomial $f=\varphi_1\cdots\varphi_k\in\PPk{\trop(M)}$ in terms of the rank functions $\rank_M,\rank_N$ of $M,N$ such that $f\cdot\trop(M)=\trop(N)$: The rational functions $\varphi_i$ are linear on the cones of $\mathcal{B}(M)$ and satisfy for all flats $F$ of $M$ that $\varphi_i(V_F)=-1$, if  $\rank_M(F)-\rank_N(F)\geq i$, and $\varphi_i(V_F)=0$ otherwise. 
\end{remark}

The rest of the section is devoted to show that the (surjective) morphism of theorem \ref{surjective} is an isomorphism in some cases. Unfortunately, so far we have not been able prove this in general.

\begin{proposition}
\label{codim1bijection}
Let $d:= \dim (\trop(M)/L)$. Then the following is an isomorphism: \[\PPOaff{1}{\trop(M)/L}/\MM\PPOaff{0}{\trop(M)/L}\rightarrow\Zaff{d-1}(\trop(M)/L), \ \ f\mapsto f\cdot\trop(M)/L.\]
\end{proposition}
\begin{proof}
It remains to prove injectivity. We can assume without loss of generality that $\{a\}$ is a flat in $M$ for every $a\in E$. Let $\varphi$ be a rational fan function with $\varphi\cdot\trop(M)=0$. The star $\Star_{\trop(M)}(p)$ around each point $p$ in the relative interior of a maximal cone of $\mathcal{B}(M)$ is isomorphic to $\R^{\dim \trop(M)}$; therefore, the locality of the intersection product (remark \ref{locality}) and the $\R^n$ case (theorem \ref{ausschneiden}) imply that $\varphi$ is linear around $p$. We can thus assume that $\varphi$ is linear on the cones of $\mathcal{B}(M)$. 
It is sufficient to show (by induction on the rank of $F$) that $\varphi(V_F)=\sum_{a\in F}\varphi(V_{\{a\}})$ for all flats $F$ of $M$. The claim is trivial for flats of rank $0$ and $1$. For a flat $F$ of rank $i\geq 2$, we choose a chain of flats 
$\emptyset=F_0\subsetneq F_1\subsetneq\ldots\subsetneq F_{i-2}\subsetneq F \subsetneq F_{i+1}\subsetneq\ldots\subsetneq F_{\rank(M)}= E$, where each $F_j$ has rank $j$. The corresponding codimension $1$ cone $\tau$ has has weight
\[0=\sum_{\stackrel{G: \text{ flat in } M}{\text{\tiny{$F_{i-2}\subsetneq G\subsetneq F$}}}}\varphi(V_G)-\varphi(V_F)-(|\{G: \text{ flat in } M, F_{i-2}\subsetneq G\subsetneq F\}|-1)\cdot \varphi(V_{F_{i-2}})\] in $\varphi\cdot\trop(M)=0$. Therefore, the claim follows by induction.
The $\trop(M)/L$ case is an immediate consequence of the $\trop(M)$ case.
\end{proof}

\begin{proposition}
\label{dim0bijection}
Let $X$ be a locally irreducible fan cycle of dimension $d$ which is connected in codimension $1$ (i.e.\ any two maximal cones are connected via a sequence of maximal cones such that the intersection of consecutive cones has codimension $1$). Then \[\PPOaff{d}{X}/\MPPOaff{d}{X}\rightarrow \Zaff{0}(X)=\Z, \ f\mapsto f\cdot X\] is an injective morphism of groups. As matroid varieties modulo lineality spaces are locally irreducible and connected in codimension 1 (this follows from \cite[lemma 2.4]{francoisrau}), the above is an isomorphism of groups if $X=\trop(M)/L$.
\end{proposition}

For a proof we need the following two lemmas:

\begin{lemma}
\label{Psisigma}
Let $\curlyx$ be a unimodular fan structure of a fan cycle $X$ of dimension $d$. Let $\sigma\in\curlyx$ be a maximal cone. Then $\Psi_{\sigma}\cdot X=\omega_{\curlyx}(\sigma)\cdot \{0\}$.
\end{lemma}

\begin{proof}
Let $v_1,\ldots,v_d$ be the primitive integral vectors generating the rays of $\sigma$. It follows from the definition of $\Psi_{v_i}$ and the intersection product with a rational function that the weight of the cone $\langle v_1,\ldots v_{i-1}\rangle$ in $\Psi_{v_{i}}\cdots\Psi_{v_d}\cdot \curlyx$ is equal to the weight of $\langle v_1,\ldots v_{i}\rangle$ in $\Psi_{v_{i+1}}\cdots\Psi_{v_d}\cdot \curlyx$. This implies the claim.
\end{proof}

\begin{lemma}
\label{diffPsisigma}
Let $\curlyx$ be a unimodular fan structure of a fan cycle $X$ of dimension $d$. Let $\sigma_1,\sigma_2\in\curlyx^{(d)}$ having a common face $\tau\in\curlyx^{(d-1)}$. 
If $X$ is locally irreducible then 
\[\omega_\curlyx(\sigma_2)\cdot\Psi_{\sigma_1}-\omega_\curlyx(\sigma_1)\cdot\Psi_{\sigma_2}=l\cdot\Psi_{\tau},\] for some linear function $l$ on $X$.   
\end{lemma}

\begin{proof}
Let $\sigma_3,\ldots,\sigma_k>\tau$ be the remaining maximal cones in $\curlyx$ adjacent to $\tau$. Let $v_1,\ldots,v_{d-1}, w_1,\ldots,w_k$ be the primitive integral vectors such that $\tau=\langle v_1,\ldots,v_{d-1}\rangle$ and $\sigma_i=\langle v_1,\ldots,v_{d-1},w_i\rangle$. As \[\omega_\curlyx(\sigma_2)\cdot\Psi_{\sigma_1}-\omega_\curlyx(\sigma_1)\cdot\Psi_{\sigma_2}=\Psi_{\tau}\cdot(\omega_\curlyx(\sigma_2)\cdot\Psi_{w_1}-\omega_\curlyx(\sigma_1)\cdot\Psi_{w_2}),\] we need a linear function $l$ satisfying \[l_{\mid \sigma_1}=\omega_\curlyx(\sigma_2)\cdot(\Psi_{w_1})_{\mid \sigma_1}, \  l_{\mid \sigma_2}=-\omega_\curlyx(\sigma_1)\cdot(\Psi_{w_2})_{\mid \sigma_2} \text{ and } l_{\mid \sigma_i}=0 \text{ for } i\geq 3.\]
The local irreducibility of $X$ implies that $v_1,\ldots,v_d,w_3,\ldots,w_k,w_1$ are linearly independent. Thus there exists a linear function $l$ such that $l(w_1)=\omega_\curlyx(\sigma_2)$ and $l(v)=0$ for $v\in\{v_1,\ldots,v_{d-1},w_3,\ldots,w_k\}$. By the balancing condition $l(w_2)=-\omega_\curlyx(\sigma_1)$; hence $l$ satisfies the above conditions.
\end{proof}

\begin{proof}[Proof of proposition \ref{dim0bijection}]
Let $f\in\PPOaff{d}{X}$ with $f\cdot X=0$. We choose a unimodular fan structure $\curlyx$ of $X$ such that $f\in\PPOaff{d}{\curlyx}$. Then there exist $a_{\sigma}\in\Z$ such that $\overline{f} = \sum_{\sigma\in\curlyx^{(d)}} a_{\sigma}\cdot \overline{\Psi_{\sigma}}$ in $\PPOaff{d}{X}/\MPPOaff{d}{X}$. Fix a maximal cone $\alpha\in\curlyx$. Since $\curlyx$ is connected in codimension $1$ it follows by lemma \ref{diffPsisigma} that $\overline{\Psi_{\sigma}}=\frac{\omega_{\curlyx}(\sigma)}{\omega_{\curlyx}(\alpha)}\cdot \overline{\Psi_{\alpha}}$ for all maximal cones $\sigma$. Hence $\overline{f}=\left(\sum_{\sigma\in\curlyx^{(d)}} a_{\sigma}\cdot\frac{\omega_{\curlyx}(\sigma)}{\omega_{\curlyx}(\alpha)}\right)\overline{\Psi_{\alpha}}$, and lemma \ref{Psisigma} implies that $\overline{f}=0$. 
\end{proof}

We can prove the following corollary in a similar way as theorem \ref{ausschneiden}.
\begin{corollary}
\label{poincaresmooth}
Let $X$ be a smooth tropical cycle and $k\in\{1,\dim X\}$. Then the following is an isomorphism of groups:
\[\Ak{k}{X}\rightarrow\Zk{\dim X-k}{X}, \ \ f\mapsto f\cdot X.\]
\end{corollary}

\begin{proof}
The injectivity follows directly from the local statement (proposition \ref{codim1bijection} resp.\ref{dim0bijection}). Let $C\in\Zk{\dim X-k}{X}$. We choose an open cover $\{V_i^j\}$ of $X$ such that for all $i,j$ we have $V_i^j\subseteq U_i$ and the weighted set $\phi_i(C\cap V_i^j)$ corresponds to (the translation of) an open tropical fan in $\phi_i(V_i^j)$. As the tropical fan associated to $\phi_i(V_i^j)$ is a matroid variety modulo lineality space, the local statement ensures that we can find
piecewise polynomials $f_i^j\in\PPk{\phi_i(V_i^j)}$ cutting out $\phi_i(C\cap V_i^j)$. Then $f=\{(V_i^j,f_i^j\circ\phi_i)\}\in\Ak{k}{X}$ is a cocycle with $f\cdot X=C$. As in the proof of theorem \ref{ausschneiden} the difference of two of these local functions gives a zero (local) intersection on the overlaps of the open sets, so by the local statement $f$ is indeed a cocycle. 
\end{proof}

\begin{remark}
Proving the injectivity of \[\PPk{\trop(M)/L}/\MPPOaff{k}{\trop(M)/L}\rightarrow \Zaff{\dim \trop(M)/L-k}(\trop(M)/L)\] is all that remains to be done in order to generalise corollary \ref{poincaresmooth} to arbitrary codimensions $k$. Note that we also needed the injectivity of intersecting with piecewise polynomials to prove the surjectivity in the preceding proof.
\end{remark}

We conclude by using corollary \ref{poincaresmooth} to pull back points and codimension $1$ cycles along morphisms with smooth targets. This could prove useful in enumerative geometry where point conditions are often described as pull-backs of points along certain evaluation morphisms (cf.\ eg.\ \cite[definition 4.2]{GKM}). Pull-backs of points are also crucial to define families of rational curves over smooth tropical varieties as morphisms of tropical varieties (with smooth target) all of whose fibers are smooth rational curves (plus two more technical conditions) in \cite[definition 3.1]{families}.  

\begin{remark}
Let $C$ be a codimension $k$ subcycle of a dimension $d$ cycle $Y$ satisfying $\Ak{k}{Y}\cong\Zk{d-k}{Y}$. Let $h:X\rightarrow Y$ be a morphism. We can define the pull-back of $C$ along $h$ to be $h^{\ast}C:=h^{\ast}f\cdot X$, where $f$ is the (unique) cocycle satisfying $f\cdot Y=C$. If $X$ and $Y$ are smooth, this coincides with the pull-back of cycles defined in \cite[definition 8.1]{francoisrau}. Furthermore, pull-backs defined in this way clearly have the properties listed in \cite[example 8.2, theorem 8.3]{francoisrau}. 
In particular, we can define pull-backs of points and codimension $1$ cycles if $Y$ is smooth, as well as pull-backs of arbitrary cycles if $Y=\R^n$.  
\end{remark}

\end{document}